\documentclass[12pt]{article}
\usepackage{amsmath,amsthm,amsfonts,amssymb,amscd,mathtools}

\allowdisplaybreaks[4]
\newtheorem {Lemma}{Lemma}[section]
\newtheorem {Theorem} {Theorem}[section]
\newtheorem {Corollary}{Corollary}[section]

= 0.0in \topmargin = 0.3in \oddsidemargin=0.1in
\begin{document}

\title{On the $\alpha$-spectral radius of uniform hypergraphs}

\author{Haiyan Guo\footnote{Email: 494565762@qq.com},
Bo Zhou\footnote{Corresponding author. E-mail: zhoubo@scnu.edu.cn}\\
School of  Mathematical Sciences, South China Normal University, \\
Guangzhou 510631, P.R. China
}

\date{}
\maketitle

\begin{abstract}
For $0\le\alpha<1$ and a uniform hypergraph $G$, the $\alpha$-spectral radius of $G$ is the largest $H$-eigenvalue of
$\alpha \mathcal{D}(G) +(1-\alpha)\mathcal{A}(G)$,
 where $\mathcal{D}(G)$ and $\mathcal{A}(G)$ are the diagonal tensor of degrees and the adjacency tensor of $G$, respectively.
We give upper bounds for the $\alpha$-spectral radius of a uniform hypergraph, propose some transformations that increase the $\alpha$-spectral radius, and determine the unique hypergraphs with maximum $\alpha$-spectral radius in some classes of uniform hypergraphs.
\end{abstract}

\section{Introduction}

Let $G$ be a hypergraph on $n$ vertices with vertex set $V(G)$ and edge set $E(G)$. If  $|e|=k$ for each $e\in E(G)$, then $G$ is said to be a $k$-uniform hypergraph. For a vertex $v\in V(G)$,
 the set of the edges containing $v$ in $G$ is denoted by $E_G(v)$, and
the degree of  $v$ in $G$, denoted by  $d_G(v)$ or $d_v$, is  the size of $E_G(v)$. 
We say that $G$ is regular if all vertices of $G$ have equal degrees. Otherwise, $G$ is irregular.

For $u, v\in V(G)$, a  walk  from $u$ to $v$ in $G$ is defined to be an alternating  sequence of  vertices and edges
$(v_0,e_1,v_1,\dots,v_{s-1},e_s, v_s)$ with $v_0=u$ and $v_s=v$ such that
edge $e_i$ contains vertices
$v_{i-1}$ and $v_i$, and  $v_{i-1}\ne v_i$ for $i=1,\dots,s$. The value $s$ is the length of this walk.
A path is a walk with all $v_i$ distinct and all $e_i$ distinct.
A cycle is a walk containing at least
two edges, all $e_i$ are distinct and all $v_i$ are distinct except $v_0=v_s$.
If there is a path from $u$ to $v$ for any $u,v\in V(G)$, then we say that $G$ is connected.
A hypertree is a connected hypergraph with no cycles. For $k\ge 2$, the number of vertices of a
$k$-uniform hypertree with $m$ edges is $1+(k-1)m$.


The distance between vertices $u$ and $v$ in a connected hypergraph $G$ is the length of a shortest path from $u$ to $v$ in  $G$. The diameter of connected hypergraph $G$ is the maximum distance between any two vertices of $G$.

For positive integers $k$ and $n$, a  tensor $\mathcal{T}=(T_{i_1\dots i_k})$ of order $k$ and dimension $n$ is a multidimensional array
with entries $\mathcal{T}_{i_1\dots i_k}\in \mathbb{C}$ for $i_j \in [n]=\{1, \dots, n\}$ and  $j\in[k]$,  where $\mathbb{C}$ is the complex field.

Let $\mathcal{M}$  be a  tensor of order $k\geq 2$  and  dimension $n$, and $\mathcal{N}$ a tensor of order $\ell\ge 1$ and  dimension $n$.
The product $\mathcal{M}\mathcal{N}$ is the tensor of order $(k-1)(\ell-1)+1$ and dimension $n$ with entries \cite{Sh}
\[
(\mathcal{M}\mathcal{N})_{ij_{1}\dots j_{k-1}}=\sum_{i_{2},\dots ,i_{k}\in [n]}\mathcal{M}_{ii_{2}\dots i_{k}}\mathcal{N}_{i_{2}j_{1}}\cdots \mathcal{N}_{i_{k}j_{k-1}},
\]
with $i\in [n]$ and $j_{1},\dots ,j_{k-1}\in [n]^{\ell-1}$. Then
for  a  tensor $\mathcal{T}$ of order $k$ and dimension $n$ and  an $n$-dimensional vector  $x=(x_1,\ldots, x_n)^\top$, 
$\mathcal{T}x$ is  an $n$-dimensional vector whose $i$-th entry is
\[(\mathcal{T}x)_i=\sum_{i_2,\ldots, i_k=1}^n \mathcal{T}_{ii_2\dots i_k}x_{i_2}\cdots x_{i_k},\]
where $i\in[n]$. Let $x^{[r]}=(x_1^r,\ldots,x_n^r)^\top$.
For some complex $\lambda$, if there is a nonzero vector $x$ such that
\[\mathcal{T}x=\lambda x^{[k-1]},\]
then $\lambda$ is called an eigenvalue of $\mathcal{T}$, and $x$ is called an eigenvector of $\mathcal{T}$ corresponding to $\lambda$. Moreover, if both $\lambda$ and $x$ are real, then we call $\lambda$ an  $H$-eigenvalue and $x$ an $H$-eigenvector of $\mathcal{T}$. See \cite{Lim,Qi05,Qi} for more details. The spectral radius  of $\mathcal{T}$ is the largest modulus of its eigenvalues, denoted by
$\rho(\mathcal{T})$.

Let $G$ be a  $k$-uniform hypergraph with vertex set $V(G)=\{v_1, \dots, v_n\}$, where $k\ge 2$.
The adjacency tensor of $G$  is defined in \cite{CD} as the tensor $\mathcal{A}(G)$ of order $k$ and dimension $n$  whose $(i_1,\dots, i_k)$-entry is $\frac{1}{(k-1)!}$ if $ \{v_{i_1}, \dots, v_{i_k}\}\in E(G)$, and $0$ otherwise.
The degree tensor of $G$ is the diagonal tensor $\mathcal{D}(G)$ of order $k$ and dimension $n$ with  $(i,\dots, i)$-entry to be the degree of vertex $v_i\in[n]$.
Then $\mathcal{\mathcal{Q}}(G)=\mathcal{D}(G)+\mathcal{A}(G)$ is the signless Laplacian tensor of $G$~\cite{Qi}. 
Motivated by work of Nikiforov~\cite{Ni1} (see also \cite{GZ,NP}), 
Lin et al.~\cite{LGZ}  proposed to study the convex linear combinations $\mathcal{A}_{\alpha}(G)$
of $\mathcal{D}(G)$ and $\mathcal{A}(G)$ defined by
\[
\mathcal{A}_{\alpha}(G)=\alpha \mathcal{D}(G) +(1-\alpha)\mathcal{A}(G),
\]
where $0\le \alpha<1$. The $\alpha$-spectral radius of $G$ is the spectral radius of $\mathcal{A}_{\alpha}(G)$, denoted by $\rho_{\alpha}(G)$. Note that $\rho_0(G)$ is the spectral radius of $G$, while $2\rho_{1/2}(G)$ is the signless Laplacian spectral radius of $G$.

For $k\ge 2$, let $G$ be a $k$-uniform hypergraph with $V(G)=[n]$, and $x$ a $n$-dimensional column
vector.
Let $x_V=\prod_{v\in V}x_v$ for $V\subseteq V(G)$.
Then
\[
x^{\top}(\mathcal{A}_{\alpha}(G)x)=\alpha\sum_{u\in V(G)}d_ux_u^{k}+(1-\alpha) k\sum_{e\in E(G)}x_e,
\]
or equivalently,
\[x^{\top}(\mathcal{A}_{\alpha}(G)x)=\sum_{e\in E(G)} \left(\alpha\sum_{u\in e}x_u^k+(1-\alpha)kx_e\right).
\]


%
%

For a uniform hypergraph $G$, bounds for  the spectral radius $\rho_0(G)$  have been given in \cite{CD,LMZW,LZM,YZL}, and bounds for the signless Laplacian spectral radius  $2\rho_{1/2}(G)$  may be found in \cite{HQX,LMZW, QSW}. Recently, Lin et al.~\cite{LGZ} gave upper bounds for $\alpha$-spectral radius of connected irregular $k$-uniform hypergraphs, extending some known bounds for ordinary graphs.
Some hypergraph transformations have been proposed to investigate the change of the $0$-spectral radius, and the unique hypergraphs that maximize or  minimize the $0$-spectral radius have been determined among some classes of uniform hypergraphs (especially for hypertrees), see, e.g., \cite{Guo2, FT,LiSQ,Ou,XW,XW2,YSS,ZSB}.

In this paper,
we give upper bounds for the $\alpha$-spectral radius of a uniform hypergraph,
propose some hypergraph transformations that increase the $\alpha$-spectral radius, and determine the unique hypergraphs with maximum $\alpha$-spectral radius in some classes of uniform hypergraphs such as the class of $k$-uniform hypercacti with $m$ edges and $r$ cycles for $0\le r\le \lfloor\frac{m}{2}\rfloor$, and  the class of $k$-uniform hypertrees with $m$ edges and diameter $d\ge 3$.

\section{Preliminaries}

 A tensor $\mathcal{T}$ of order $k\ge 2$ and dimension $n$ is said to be weakly reducible, if there is a nonempty proper subset $J$ of $[n]$ such that for  $i_1\in J$ and $i_j\in [n]\setminus J$ for some $j=2, \dots, k$, $T_{i_1 \dots i_k}=0$. Otherwise,  $\mathcal{T}$  is weakly irreducible.

For $k\ge 2$, an $n$-dimensional vector $x$ is said to be $k$-unit if $\sum_{i=1}^nx_i^k=1$.

\begin{Lemma}\label{irreducibility}\cite{FGH,YY}
Let $\mathcal{T}$ be a  nonnegative tensor of order $k\ge 2$ and diminsion $n$. Then $\rho(\mathcal{T})$ is an eigenvalue of $\mathcal{T}$ and there is a $k$-unit nonnegative eigenvector corresponding to $\rho(\mathcal{T})$.
If furthermore $\mathcal{T}$ is weakly irreducible, then there is
a unique $k$-unit positive eigenvector corresponding to $\rho(\mathcal{T})$.
\end{Lemma}

If  $G$ is a $k$-uniform hypergraph with $k\ge 2$, then $\mathcal{A}_{\alpha}(G)$  is weakly irreducible if and only if $G$ is connected (see \cite{PZ,Qi} for the treatment of $\mathcal{A}_0(G)$ and $2\mathcal{A}_{1/2}(G)$ respectively). Thus,  if $G$ is connected, then by Lemma~\ref{irreducibility}, there is
a unique  $k$-unit positive $H$-eigenvector $x$ corresponding to  $\rho_{\alpha}(G)$,
 which is called the
 $\alpha$-Perron vector   of $G$.

For  a nonnegative tensor  $\mathcal{T}$ of order $k\ge 2$ and dimension $n$, let $r_i(\mathcal{T})=\sum_{i_2\dots i_k=1}^n \mathcal{T}_{ii_2\dots i_k}$ for $i=1, \dots, n$.

\begin{Lemma}\label{Tensor1} \cite{LCL,YY}
Let $\mathcal{T}$ be a nonnegative tensor of order $k\ge 2$ and dimension $n$. Then
\[
\rho(\mathcal{T})\leq \max_{1\leq i \leq n}r_i(\mathcal{T})\]
with equality when  $\mathcal{T}$ is weakly irreducible if and only if $r_1(\mathcal{T})=\cdots=r_n(\mathcal{T})$.
\end{Lemma}
For two tensors $\mathcal{M}$ and $\mathcal{N}$ of order $k\ge 2$ and dimension $n$, if  there is an $n\times n$ nonsingular diagonal matrix $U$ such that $\mathcal{N}=U^{-(k-1)}\mathcal{M}U$, then we say that $\mathcal{M}$ and $\mathcal{N}$ are  diagonal similar.

\begin{Lemma}\cite{Sh} \label{Tensor2}
Let $\mathcal{M}$ and $\mathcal{N}$ be two diagonal similar tensors of order $k\ge 2$ and dimension $n$. Then $\mathcal{M}$  and $\mathcal{N}$
have the same  real eigenvalues.
\end{Lemma}

Let $G$ be a connected $k$-uniform hypergraph on $n$ vertices, where $k\ge 2$. Let $0\le \alpha<1$.  For an $n$-dimensional $k$-unit nonnegative vector $x$,  by \cite[Theorem~2]{Qi13} (and its proof) and  Lemma~\ref{irreducibility},  we have
$\rho_{\alpha}(G)\ge x^\top (A_{\alpha}(G)x)$ with equality if and only if $x$ is the $\alpha$-Perron vector of $G$. If $x$ is the $\alpha$-Perron vector of $G$, then
 for any $v\in V(G)$,
\[
\rho_{\alpha}(G)x_v^{k-1}
=\alpha d_vx_v^{k-1}+(1-\alpha)\sum_{e\in E_v(G)} x_{e\setminus\{v\}},
\]
which is called the eigenequation of $G$ at $v$.

For a hypergraph $G$ with  $\emptyset\ne X\subseteq V (G)$, let $G[X]$ be the subhypergraph induced by
$X$, i.e., $G[X]$ has vertex set $X$ and edge set $\{e\subseteq X: e\in E(G)\}$.  If $E'\subseteq E(G)$, then $G-E'$ is the hypergraph obtained from $G$ by deleting the edges in $E'$. If $E'$ is set  of subsets of $V(G)$ and no element of $E'$ is an edge of $G$, then $G+E'$ is the hypergraph obtained from $G$ by adding elements of $E'$ as edges.

A $k$-uniform hypertree with $m$ edges is a hyperstar, denoted by $S_{m,k}$, if  all  edges share a common vertex.
A $k$-uniform loose path with $m\ge 1$ edges, denoted by $P_{m,k}$, is the $k$-uniform hypertree whose vertices and edges  may be labelled as
 $(v_0, e_1, v_1,\dots, v_{m-1}, e_m, v_m)$ such that the vertices $v_1,\dots, v_{m-1}$ are of degree $2$, and all the other vertices of $G$ are of degree $1$.

 If $P$ is a path or a cycle of a hypergraph $G$, $V(P)$ denotes the vertex set of the hypergraph $P$.

\section{Upper bounds for $\alpha$-spectral radius}

For  a connected irregular $k$-uniform hypergraph $G$ with $n$ vertices, maximum degree $\Delta$ and diameter $D$, where $2\le k<n$, it was shown in \cite{LGZ} that for $0\le \alpha<1$,
\[
\rho_\alpha(G)<\Delta-\frac{4(1-\alpha)}{\left((4D-1-2\alpha)(k-1)+1\right)n}.
\]
For a $k$-uniform hypergraph $G$, upper bounds on $\rho_0(G)$ and $2\rho_{1/2}(G)$ have been given in \cite{YZL,LMZW}.

\begin{Theorem} \label{Mo} Let $G$ be a $k$-uniform  hypergraph on $n$ vertices with maximum
degree $\Delta$ and second maximum degree $\Delta'$, where $k\ge 2$.

For  $\alpha=0$, let $\delta=\left(\frac{\Delta}{\Delta'}\right)^{\frac{1}{k}}$, and for $0<\alpha<1$,
let $\delta=1$ if $\Delta=\Delta'$ and $\delta$ be a root of $h(t)=0$ in $((\frac{\Delta}{\Delta'})^{\frac{1}{k}}, + \infty)$
if $\Delta>\Delta'$, where $h(t)=(1-\alpha)\Delta't^{k}+\alpha(\Delta'-\Delta)t^{k-1}-(1-\alpha)\Delta$ for $0\le \alpha <1$. Then
\begin{equation} \label{MZh}
\rho_{\alpha}(G)\le \alpha \Delta+(1-\alpha)\Delta\delta^{-(k-1)}.
\end{equation}
 Moreover, if $G$ is connected, then equality holds in $(\ref{MZh})$ if and only if $G$ is a regular hypergraph or  $G\cong G'$, where $V(G')=V(H)\cup \{v\}$, $E(G')=\{e\cup \{v\}: e\in E(H)\}$, and $H$ is
a  regular $(k-1)$-uniform hypergraph on $n-1$ vertices with $v\notin V(H)$.
\end{Theorem}

\begin{proof}  By Theorem~2.1 and Lemma~2.2 in \cite{Sh}, we may assume that $d_1\ge \dots\ge d_n$. Then $\Delta=d_1$ and $\Delta'=d_2$.

 If $d_{1}=d_{2}$, then $\delta=1$, and by Lemma~\ref{Tensor1}, we have
\[
\rho_{\alpha}(G)\le\max_{1\leq i\leq n}r_{i}(\mathcal{A}_{\alpha}(G))=\max_{1\leq i\leq n}d_{i}=d_{1}= \alpha d_{1}+(1-\alpha)d_{1}\delta^{-(k-1)},
\]
and when $G$ is connected, $\mathcal{A}_{\alpha}(G)$ is weakly irreducible, thus by Lemma~\ref{Tensor1}, equality holds (\ref{MZh}) if and only if $r_1(\mathcal{A}_{\alpha}(G))=\dots =r_n(\mathcal{A}_{\alpha}(G))$, i.e., $G$ is a regular hypergraph.

Suppose in the following   that $d_{1}> d_{2}$. Let $U=\mbox{diag}(t,1,\dots,1)$ be an $n\times n$ diagonal matrix, where $t>1$ is a variable to be determined later. Let $\mathcal{ T}=U^{-(k-1)}\mathcal{A}_{\alpha}(G)U$. By Lemma~\ref{Tensor2}, $\mathcal{A}_{\alpha}(G)$ and $\mathcal{T}$ have the same real eigenvalues. Obviously, both $\mathcal{A}_{\alpha}(G)$ and $\mathcal{T}$ are nonnegative tensors of order $k$ and dimension $n$. By Lemma~\ref{irreducibility}, $\rho(\mathcal{A}_{\alpha}(G))$ is an eigenvalue of $\mathcal{A}_{\alpha}(G)$  and $\rho(\mathcal{T})$ is an eigenvalue of $\mathcal{T}$. Therefore $\rho_{\alpha}(G)=\rho(\mathcal{A}_{\alpha}(G))=\rho(\mathcal{T})$. For $i\in [n]\setminus \{1\}$, let
$d_{1,i}=|\{e: 1,i\in e\in E(G)\}|$. Obviously, $d_{1,i}\le d_i$.
Note that
\begin{eqnarray*}\begin{split}
r_{1}(\mathcal{T})
=&\sum_{i_{2},\ldots,i_{k}\in[n]}\mathcal{T}_{1i_{2}\dots i_{k}}\\
=&\alpha \mathcal{D}_{1\dots 1}+(1-\alpha)\sum_{i_{2},\ldots,i_{k}\in[n]}U_{11}^{-(k-1)}\mathcal{A}_{1i_{2}\ldots i_{k}}U_{i_{2}i_{2}}\cdots U_{i_{k}i_{k}}\\
=&\alpha d_{1}+(1-\alpha)\sum_{i_{2},\ldots,i_{k}\in[n]\setminus\{1\}}\frac{1}{t^{k-1}}\mathcal{A}_{1i_{2}\ldots i_{k}}\\
=&\alpha d_{1}+\frac{(1-\alpha)d_{1}}{t^{k-1}},
\end{split}
\end{eqnarray*}
and for $2\leq i\leq n$,
\begin{eqnarray*}
\begin{split}
r_{i}(\mathcal{T})
=&\sum_{i_{2},\dots,i_{k}\in[n]}\mathcal{T}_{ii_{2}\dots i_{k}}\\
=& \alpha \mathcal{D}_{i\dots i}+(1-\alpha)\sum_{i_{2},\dots,i_{k}\in[n]}U_{ii}^{-(k-1)}\mathcal{A}_{ii_{2}\ldots i_{k}}U_{i_{2}i_{2}}\cdots U_{i_{k}i_{k}}\\
=& \alpha d_{i}+(1-\alpha)\sum_{i_{2},\dots,i_{k}\in[n]\atop 1\in\{i_{2},\ldots,i_{k}\}}U_{ii}^{-(k-1)}\mathcal{A}_{ii_{2}\ldots i_{k}}U_{i_{2}i_{2}}\cdots U_{i_{k}i_{k}}\\
&+(1-\alpha)\sum_{i_{2},\dots,i_{k}\in[n]\atop 1\not\in\{i_{2},\ldots,i_{k}\}}U_{ii}^{-(k-1)}\mathcal{A}_{ii_{2}\ldots i_{k}}U_{i_{2}i_{2}}\cdots U_{i_{k}i_{k}}\\
=& \alpha d_{i}+(1-\alpha)\sum_{i_{2},\dots,i_{k}\in[n]\atop 1\in\{i_{2},\ldots,i_{k}\}}\mathcal{A}_{ii_{2}\ldots i_{k}}t+(1-\alpha)\sum_{i_{2},\dots,i_{k}\in[n]\atop 1\not\in\{i_{2},\ldots,i_{k}\}}\mathcal{A}_{ii_{2}\ldots i_{k}}\\
=& \alpha d_i+(1-\alpha)td_{1,i}+(1-\alpha)(d_{i}-d_{1,i})\\
=&  d_i+(1-\alpha)(t-1)d_{1,i}\\
\le& (1+(1-\alpha)(t-1))d_{i}\\
\le& (1+(1-\alpha)(t-1))d_{2}
\end{split}
\end{eqnarray*}
with equality if and only if $d_{1,i}=d_i=d_2$.

Note that $h((\frac{d_{1}}{d_{2}})^{\frac{1}{k}})=\alpha(d_2-d_1)(\frac{d_{1}}{d_{2}})^{\frac{k-1}{k}}\le 0$ with equality if and only if $\alpha=0$,  and that $h(+\infty)>0$. Thus  $h(t)=0$ does have a root  $\delta$, as required. Let  $t=\delta$. Then $t>1$,
\[
\alpha d_{1}+\frac{(1-\alpha)d_{1}}{t^{k-1}}=(1+(1-\alpha)(t-1))d_{2},
\]
and thus
 for $1\le i\leq n$,
\[
r_{i}(\mathcal{T})\le \alpha d_{1}+(1-\alpha)d_{1}\delta^{-(k-1)}.
\]
Now by Lemma~\ref{Tensor1},
\[
\rho_{\alpha}(G)=\rho(\mathcal{T})\leq \max_{1\leq i\leq n}r_{i}(\mathcal{T})\leq \alpha d_{1}+(1-\alpha)d_{1}\delta^{-(k-1)}.
\]
This proves (\ref{MZh}).

Suppose that $G$ is connected. Then  $\mathcal{A}_{\alpha}$ is weakly irreducible, and so is $\mathcal{T}$.

Suppose that equality holds in (\ref{MZh}). From the above arguments and by Lemma~\ref{Tensor1}, we have $r_1(\mathcal{T})=\cdots=r_n(\mathcal{T})= \alpha d_{1}+(1-\alpha)d_{1}\delta^{-(k-1)}$, and  $d_{1,i}=d_{i}=d_2$ for $i=2, \dots, n$. Then  vertex $1$ is contained in each edge of $G$. Let $H$ be the hypergraph with $V(H)=V(G)\setminus\{1\}=\{2, \dots, n\}$ and $E(H)=\{e\setminus\{1\}: e\in E(G)\}$. Then $H$ is a  regular $(k-1)$-uniform hypergraph on vertices $2, \dots, n$  of degree $d_2$. Therefore $G\cong G'$, where $V(G')=V(H)\cup \{1\}$, $E(G')=\{e\cup \{v\}: e\in E(H)\}$, and $H$ is
a  regular $(k-1)$-uniform hypergraph on vertices $2, \dots, n$  of degree $d_2$.


Conversely,  if $G\cong G'$, where $V(G')=V(H)\cup \{1\}$, $E(G')=\{e\cup \{v\}: e\in E(H)\}$, and $H$ is
a  regular $(k-1)$-uniform hypergraph on vertices $2, \dots, n$  of degree $d_2$, then by the above arguments, we have $r_{i}(\mathcal{T})=\alpha d_{1}+(1-\alpha)d_{1}(\frac{1}{\delta})^{k-1}$ for $1\leq i\leq n$, and thus by Lemma~\ref{Tensor2},
 $\rho(\mathcal{A}_{\alpha}(G))=\rho(\mathcal{T})=\alpha d_{1}+(1-\alpha)d_{1}\delta^{-(k-1)}$, i.e., (\ref{MZh}) is an equality.
\end{proof}

As $\delta\ge (\frac{d_{1}}{d_{2}})^{\frac{1}{k}}$ with equality if and only if $d_1=d_2$, we have by Theorem~\ref{Mo} that $\rho_{\alpha}(G)\le \alpha d_1+(1-\alpha)d_1^{\frac{1}{k}}d_2^{1-\frac{1}{k}}$ with equality if and only if $G$ is regular.

Letting $\alpha=0$ in Theorem~\ref{Mo}, we have $\delta=\left(\frac{d_1}{d_2}\right)^{\frac{1}{k}}$ and thus (\ref{MZh}) becomes $\rho_0(G)\le d_1^{\frac{1}{k}}d_2^{1-\frac{1}{k}}$, see \cite{YZL}. Letting $\alpha=\frac{1}{2}$ in Theorem~\ref{Mo}, $\delta$ is the  root of  $d_2t^k+(d_2-d_1)t^{k-1}-d_1=0$, and  (\ref{MZh}) becomes $2\rho_{1/2}(G)\le d_{1}+d_{1}\delta^{-(k-1)}$, see \cite{LMZW}.

%

Let $G$ be a connected  $k$-uniform hypergraph with $n$ vertices, $m$ edges, maximum degree $\Delta$ and diameter $D$, where $k\ge 2$. For $0\le \alpha<1$, let $\overline{x}$ be the maximum entry of the $\alpha$-Perron vector of $G$.
From \cite{LGZ}, we have
\[
\rho_\alpha(G)\le \Delta-\frac{(1-\alpha)k(n\Delta-km)}{2(n\Delta-km)(k-1)D+(1-\alpha)k}\overline{x}^k,
\]
and if $D=1$ and $k\ge 3$, then
\[
\rho_\alpha(G)\le  \Delta-\frac{(1-\alpha)(n\Delta-km)n}{2(n\Delta-km)(k-1)+(1-\alpha)n}\overline{x}^k.
\]

\begin{Theorem}\label{x_max} Let $G$ be a connected $k$-uniform hypergraph on $n$ vertices with $m$ edges and  maximum degree $\Delta$, where $k\ge 2$.  Let $x$ be the $\alpha$-Perron vector of $G$ with maximum entry $\overline{x}$.
 For $0\le \alpha<1$, we have
 \[
 \rho_{\alpha}(G)\le \alpha \Delta+(1-\alpha)km\overline{x}^k
 \]
\[
\rho_{\alpha}(G)\le \alpha \Delta+(1-\alpha)\left(\sum_{i\in V(G)}d_i^{\frac{k}{k-1}}\right)^{\frac{k-1}{k}}\overline{x}^{k-1}
\]
 with either equality if and only if $G$ is regular.
\end{Theorem}

\begin{proof} From the eigenequation of $G$ at $i\in V(G)$, we have
\[
(\rho_{\alpha}-\alpha \Delta)x_i^{k-1}\le (\rho_{\alpha}-\alpha d_i)x_i^{k-1}=(1-\alpha)\sum_{e\in E_i(G)}\prod_{v\in e\setminus\{i\}}x_v\le (1-\alpha)d_i \overline{x}^{k-1}
\]
with equality if and only if for $v\in e\setminus\{i\}$ with $e\in E_i(G)$, $x_v=\overline{x}$.
Then
\[
(\rho_{\alpha}-\alpha \Delta)x_i^{k}\le  (1-\alpha)d_i \overline{x}^{k},\]
and thus
\[
\rho_{\alpha}-\alpha \Delta\le (1-\alpha)\overline{x}^{k}\sum_{i\in V(G)}d_i =(1-\alpha)km\overline{x}^k
\]
with equality if and only if all entries of $x$ are equal, or equivalently, $G$ is regular.

On the other hand, we have
\[(\rho_{\alpha}-\alpha \Delta)^{\frac{k}{k-1}}x_i^{k}\le (1-\alpha)^{\frac{k}{k-1}}d_i^{\frac{k}{k-1}} \overline{x}^{k},
\]
and thus
\[
(\rho_{\alpha}-\alpha \Delta)^{\frac{k}{k-1}}\le (1-\alpha)^{\frac{k}{k-1}}\overline{x}^{k}\sum_{i\in V(G)}d_i^{\frac{k}{k-1}},
\]
implying that
\[
\rho_{\alpha}(G)\le \alpha \Delta+(1-\alpha)\left(\sum_{i\in V(G)}d_i^{\frac{k}{k-1}}\right)^{\frac{k-1}{k}}\overline{x}^{k-1}
\]
with equality if and only if $G$ is regular.
\end{proof}

Let $\alpha=0$ in Theorem~\ref{x_max}, we have $\overline{x}\ge \frac{\rho_{0}^{\frac{1}{k-1}}}{\left(\sum_{i\in V(G)}d_i^{\frac{k}{k-1}}\right)^{\frac{1}{k}}}$, which has been reported in \cite{LZB}.

\section{Transformations increasing $\alpha$-spectral radius}

In the following, we propose several types of hypergraph transformations that increase the $\alpha$-spectral radius.

\begin{Theorem}\label{moving} For $k\ge 2$, let  $G$ be a $k$-uniform hypergraph with $u, v_1, \dots, v_r\in V(G)$ and $e_1,\dots, e_r\in E(G)$ for $r\ge 1$ such that $u\notin e_i$ and $v_i\in e_i$ for $i=1,\dots, r$, where $v_1, \dots, v_r$ are not necessarily distinct. Let  $e'_i=(e_i\setminus\{v_i\})\cup\{u\}$ for $i=1,\dots, r$. Suppose that $e_i'\not \in E(G)$ for $i=1, \dots, r$. Let $G'=G-\{e_1,\dots, e_r \}+\{e'_1,\dots, e'_r\}$. Let $x$  the $\alpha$-Perron vector  of $G$. If $x_u\ge \max\{x_{v_1}, \dots, x_{v_r}\}$, then $\rho_{\alpha}(G')>\rho_{\alpha}(G)$ for $0\le \alpha<1$.
\end{Theorem}

\begin{proof} Note that  $\rho_{\alpha}(G)=x^{\top}(\mathcal{A}_{\alpha}(G)x)$ and $\rho_{\alpha}(G')\ge x^{\top}(\mathcal{A}_{\alpha}(G')x)$ with equality if and only if $x$  is also the $\alpha$-Perron vector  of $G'$.
Thus
\begin{eqnarray*}\begin{split}
\rho_{\alpha}(G')-\rho_{\alpha}(G)
\ge & x^{\top}(\mathcal{A}_{\alpha}(G')x)-x^{\top}(\mathcal{A}_{\alpha}(G)x)\\
=&  \alpha \left(rx_u^{k}-\sum_{ i=1}^rx_{v_{i}}^{k}\right)+(1-\alpha)k\sum_{i=1}^r(x_{u}-x_{v_{i}})x_{e_{i}\setminus\{v_i\}}\\
\ge&0,
\end{split}
\end{eqnarray*}
and thus $\rho_{\alpha}(G')\ge \rho_{\alpha}(G)$.
Suppose that  $\rho_{\alpha}(G')=\rho_{\alpha}(G)$. Then $\rho_{\alpha}(G')= x^{\top}(\mathcal{A}_{\alpha}(G')x)$, and thus $x$ is the $\alpha$-Perron vector  of $G'$. From the eigenequations of $G'$ and $G$ at $u$ and noting that $E_u(G')=E_u(G)\cup \{e_1', \dots, e_r'\}$, we have
\begin{eqnarray*}\begin{split}
\rho_{\alpha}(G')x_{u}^{k-1}=& \alpha (d_u+r)x_u^{k-1}+(1-\alpha)\sum_{e\in E_u(G')}x_{e\setminus \{u\}}\\
>&\alpha d_ux_u^{k-1}+(1-\alpha)\sum_{e\in E_u(G)}x_{e\setminus \{u\}}\\
=& \rho_{\alpha}(G)x_{u}^{k-1},
\end{split}
\end{eqnarray*}
a contradiction. It follows that $\rho_{\alpha}(G')>\rho_{\alpha}(G)$.
\end{proof}

We say that the hypergraph $G'$ in Theorem~\ref{moving} is obtained from $G$ by moving edges $e_1, \dots, e_r$ from $v_1,\dots, v_r$ to $u$. Theorem~\ref{moving} has been established in \cite{LSQ} for $\alpha=0,\frac{1}{2}$.

\begin{Theorem}\label{switching} Let $G$ be a connected $k$-uniform hypergraph with $k\ge 2$, and $e$ and $f$ be two edges of $G$ with $e\cap f=\emptyset$. Let $x$ be the $\alpha$-Perron vector  of $G$. Let $U\subset e$ and $V\subset f$ with $1\le |U|=|V|\le k-1$.
Let $e'=U\cup (f\setminus V)$ and $f'=V\cup (e\setminus U)$. Suppose that $e', f'\notin E(G)$. Let $G'=G-\{e, f\}+\{e', f'\}$.
If $x_{U}\ge x_{V}$, $x_{e\setminus U}\le x_{f\setminus V}$ and one is strict, then $\rho_{\alpha}(G)< \rho_{\alpha}(G')$ for $0\le \alpha<1$.
\end{Theorem}

\begin{proof} 
Note that
\begin{eqnarray*}\begin{split}
\rho_{\alpha}(G')-\rho_{\alpha}(G)
\ge& x^{\top}(\mathcal{A}_{\alpha}(G')x)-x^{\top}(\mathcal{A}_{\alpha}(G)x)\\
=& (1-\alpha) k\sum_{e\in E(G')}x_e-(1-\alpha) k\sum_{e\in E(G)}x_e\\
=& (1-\alpha) k(x_{U}x_{f\setminus V}+x_{V}x_{e\setminus U}-x_{U}x_{e\setminus U}-x_{V}x_{f\setminus V})\\
=& (1-\alpha)k(x_{U}-x_{V})(x_{f\setminus V}-x_{e\setminus U})\\
\ge& 0.
\end{split}
\end{eqnarray*}
Thus $\rho_{\alpha}(G')\ge \rho_{\alpha}(G)$.
Suppose that $\rho_{\alpha}(G')=\rho_{\alpha}(G)$. Then $\rho_{\alpha}(G')=x^{\top}(\mathcal{A}_{\alpha}(G')x)$
and thus $x$ is the $\alpha$-Perron vector  of $G'$. Suppose without loss of generality that $x_{e\setminus U}<x_{f\setminus V}$. Then for $u\in U$
\[
-x_{e\setminus \{u\}}+x_{e'\setminus \{u\}}=-x_{U\setminus \{u\}}\left(x_{e\setminus U}-x_{f\setminus V}\right)>0.
\]
From the eigenequations of $G'$ and $G$  at a vertex $u\in U$, we have
\begin{eqnarray*}\begin{split}
\rho_{\alpha}(G')x_{u}^{k-1}
=& \alpha d_ux_u^{k-1}+(1-\alpha)\sum_{e\in E_u(G')}x_{e\setminus \{u\}}\\
=& \alpha d_ux_u^{k-1}+(1-\alpha)\left(\sum_{e\in E_u(G)}x_{e\setminus \{u\}}-x_{e\setminus \{u\}}+x_{e'\setminus \{u\}}\right)\\
>&\alpha d_ux_u^{k-1}+(1-\alpha)\sum_{e\in E_u(G)}x_{e\setminus \{u\}}\\
=& \rho_{\alpha}(G)x_{u}^{k-1},
\end{split}
\end{eqnarray*}
a contradiction. It follows that $\rho_{\alpha}(G')>\rho_{\alpha}(G)$.
\end{proof}

The above result has been known for $k=2$ in \cite{GZ}  and $\alpha=0$ \cite{XW2}.

A path $P=(v_0,e_1,v_1,\dots,v_{s-1},e_s, v_s)$ in a $k$-uniform hypergraph $G$ is called a pendant path at $v_0$, if $d_G(v_0)\ge 2$,  $d_G(v_i)=2$ for $1\leq i\leq s-1$, $d_G(v)=1$ for $v\in e_i\setminus\{v_{i-1}, v_i\}$ with $1\leq i\leq s$, and $d_G(v_s)=1$. If $s=1$, then  we call $P$ or $e_1$ a pendant edge of $G$ (at $v_0$). A pendant path of length $0$ at $v_0$ is understood as the trivial path consisting of a single vertex $v_0$.

If $P$ is a pendant path at $u$ in a $k$-uniform hypergraph $G$, we say $G$ is obtained from $H$ by attaching a pendant path  $P$ at $u$ with $H=G[V(G)\setminus(V(P)\setminus\{u\})]$. 
In this case, we write $G=H_u(s)$ if the length of $P$ is $s$. Let $H(u,0)=H$.

For a $k$-uniform hypergraph $G$ with $u\in V(G)$, and $p\ge q\ge 0$, let $G_u(p,q)=(G_u(p))_u(q)$.

\begin{Theorem}\label{pq}
For $k\ge 2$, let $G$ be a connected $k$-uniform hypergraph with $|E(G)|\geq1$ and $u\in V(G)$. For  $p\geq q\geq 1$ and $0\le \alpha<1$, we have $\rho_{\alpha}(G_{u}(p, q))> \rho_{\alpha}(G_{u}(p+1, q-1))$.
\end{Theorem}

\begin{proof}
Let $(u, e_1, u_1,\dots, u_{p}, e_{p+1}, u_{p+1})$ and $(u, f_1, v_1,\dots, v_{q-2}, f_{q-1}, v_{q-1})$ be the pendant paths of $G_u(p+1, q-1)$ at $u$ of lengths $p+1$ and $q-1$, respectively. Let $v_0=u$.  Let $x$ be the $\alpha$-Perron vector  of $G_u(p+1, q-1)$.

Suppose that $\rho_{\alpha}(G_{u}(p, q))<\rho_{\alpha}(G_{u}(p+1, q-1))$. We prove that $x_{u_{p-i}}>x_{v_{q-i-1}}$ for $i=0,\dots, q-1$.

Suppose that $x_{v_{q-1}}\ge x_{u_p}$.  Let $H$ be the $k$-uniform hypergraph obtained from $G_{u}(p+1, q-1)$ by moving $e_{p+1}$ from $u_{p}$ to $v_{q-1}$. By Theorem~\ref{moving} and noting that $H\cong G_u(p, q)$, we have $\rho_{\alpha}(G_{u}(p, q))=\rho_{\alpha}(H)> \rho_{\alpha}(G_{u}(p+1, q-1))$, a contradiction.   Thus
 $x_{u_p}>x_{v_{q-1}}$.

Suppose that $q\ge 2$ and  $x_{u_{p-i}}>x_{v_{q-i-1}}$,  where  $0\le i\le q-2$. We want to show that $x_{u_{p-(i+1)}}> x_{v_{q-(i+1)-1}}$. Suppose that this is not true, i.e.,
$x_{v_{q-i-2}}\ge x_{u_{p-i-1}}$. Suppose that $x_{e_{p-i}\setminus\{u_{p-i-1},u_{p-i}\}}\le  x_{f_{q-i-1}\setminus\{v_{q-i-2},v_{q-i-1}\}}$. Then $x_{e_{p-i}\setminus\{u_{p-i}\}}\le  x_{f_{q-i-1}\setminus\{v_{q-i-1}\}}$. Let $H'=G_u(p+1, q-1)-\{e_{p-i}, f_{q-i-1}\}+\{e', f'\}$, where $e'=\{u_{p-i}\}\cup (f_{q-i-1}\setminus \{v_{q-i-1}\})$ and $f'=\{v_{q-i-1}\}\cup (e_{p-i}\setminus\{u_{p-i}\})$. Obviously, $H'\cong G_{u}(p, q)$. By Theorem~\ref{switching}, we have $\rho_{\alpha}(G_{u}(p, q))=\rho_{\alpha}(H')> \rho_{\alpha}(G_{u}(p+1, q-1))$, a contradiction.
Thus  $x_{e_{p-i}\setminus\{u_{p-i-1},u_{p-i}\}}>x_{f_{q-i-1}\setminus\{v_{q-i-2},v_{q-i-1}\}}$, and then  $x_{e_{p-i}\setminus\{u_{p-i-1}\}}> x_{f_{q-i-1}\setminus\{v_{q-i-2}\}}$.
Let $H''=G_u(p+1, q-1)-\{e_{p-i}, f_{q-i-1}\}+\{e'', f''\}$, where $e''=(e_{p-i}\setminus \{u_{p-i-1})\cup \{v_{q-i-2}\}$ and $f''=f_{q-i-1}\setminus\{v_{q-i-2}\})\cup \{u_{p-i-1}\}$. Obviously, $H''\cong G_{u}(p, q)$. By Theorem~\ref{switching}, we have  $\rho_{\alpha}(G_{u}(p, q))=\rho_{\alpha}(H'')> \rho_{\alpha}(G_{u}(p+1, q-1))$, also a contradiction. It follows that $x_{u_{p-i-1}}> x_{v_{q-i-2}}$, i.e., $x_{u_{p-(i+1)}}> x_{v_{q-(i+1)-1}}$.

Therefore  $x_{u_{p-i}}>x_{v_{q-i-1}}$ for $i=0,\dots, q-1$. Particularly, $x_{u_{p-q+1}}>x_{v_0}$.

Now let $H^*$ be the $k$-uniform hypergraph obtained from  $G_{u}(p+1, q-1)$  by moving all the edges containing $u$ except $e_1$ and $f_1$ from $u$ to $u_{p-q+1}$. By Theorem~\ref{moving} and noting that $H^*\cong G_u(p, q)$, we have $\rho_{\alpha}(G_u(p, q))> \rho_{\alpha}(G_u(p+1, q-1))$, a contradiction. Therefore  $\rho_{\alpha}(G_{u}(p, q))> \rho_{\alpha}(G_{u}(p+1, q-1))$.
\end{proof}

The above result has been reported for $k=2$ in \cite{GZ} and $\alpha=0$ in \cite{XW2}.

\begin{Theorem}\label{3} Let $G$ be a $k$-uniform hypergraph with $k\ge 2$, $e=\{v_1,\dots,v_k\}$ be an edge of $G$ with $d_G(v_i)\ge 2$ for $i=1, \dots, r$, and $d_G(v_i)=1$ for $i=r+1, \dots, k$, where $3\le r\le k$. Let $G'$ be the hypergraph obtained from $G$ by moving all edges containing $v_3, \dots, v_r$ but not containing $v_1$  from $v_3, \dots, v_r$ to $v_1$. Then $\rho_{\alpha} (G')> \rho_{\alpha} (G)$ for $0\le \alpha<1$.
\end{Theorem}

\begin{proof} Let $x$ be the $\alpha$-Perron vector  of $G$, and $x_{v_t}=\max\{x_{v_i}: 3\leq i\leq r\}$. If $x_{v_1}\geq x_{v_t}$, then by Theorem~\ref{moving}, $\rho_{\alpha}(G')> \rho_{\alpha}(G)$. Suppose that $x_{v_1}< x_{v_t}$. Let  $G''$ be the hypergraph obtained from $G$ by moving all edges containing $v_i$ but not containing $v_t$ from $v_i$ to $v_t$ for all $3\leq i\leq r$ with $i\neq t$, and moving all edges  containing $v_1$ but not containing $v_t$ from $v_1$ to $v_t$. It is obvious that $G'' \cong G'$. By Theorem~\ref{moving}, we have $\rho_{\alpha}(G')=\rho_{\alpha}(G'')> \rho_{\alpha}(G)$.
\end{proof}

\section{Hypergraphs with large $\alpha$-spectral radius}


A hypercactus is a connected $k$-uniform hypergraph in which any two cycles £¨viewed as two hypergraphs) have at most one vertex in common.
Let $H_{m, r, k}$ be a $k$-uniform hypergraph consisting of  $r$ cycles of length $2$ and $m-2r$ pendant edges with a vertex in common. If $r=0$, then $H_{m, r, k}\cong S_{m,k}$. 

\begin{Theorem}\label{cyclic}  For $k\ge 2$, let $G$ be a $k$-uniform hypercactus with $m$ edges and $r$ cycles, where $0\le r\le\lfloor\frac{m}{2}\rfloor$ and $m\ge 2$. For $0\le \alpha<1$, we have $\rho_{\alpha}(G)\le\rho_{\alpha}(H_{m, r, k})$  with equality if and only if $G\cong H_{m, r, k}$.
\end{Theorem}

\begin{proof}  Let $G$ be a $k$-uniform hypercactus with maximum $\alpha$-spectral radius among $k$-uniform hypercacti with $m$ edges and $r$ cycles.

Let $x$ be the $\alpha$-Perron vector  of $G$.

Suppose first that $r=0$, i.e., $G$ is a hypertree with $m$ edges. Let $d$ be diameter of $G$. Obviously, $d\ge 2$. Suppose that $d\ge 3$. Let $(u_{0},e_{1}, u_{1}, \dots,e_{d}, u_{d})$ be a diametral path of $G$. Choose $u\in e_{d-1}$ with $x_u=\max\{x_{v}: v\in e_{d-1}\}$. Let $G_1$ be the hypertree obtained from $G$ by moving all edges (except $e_{d-1}$)   containing a vertex  of  $e_{d-1}$ different from $u$ from these vertices to $u$. By  Theorem~\ref{moving}, we have $\rho_{\alpha}(G_1)> \rho_{\alpha}(G)$, a contradiction. Thus $d=2$, implying that $G\cong S_{m,k}=H_{m,0,k}$.

Suppose in the following that $r\ge 1$.

If there exists an edge  $e$ with at least three vertices of degree at least $2$. Let  $e=\{v_1,\dots,v_k\}$ with $d_G(v_i)\ge 2$ for $i=1, \dots, \ell$, and $d_G(v_i)=1$ for $i=\ell+1, \dots, k$, where $3\le \ell\le k$. Let  $G'$ be the hypergraph  obtained from $G$ by moving all edges containing $v_3, \dots, v_{\ell}$  except $e$  from $v_3, \dots, v_{\ell}$ to $v_1$. Obviously,  $G'$ is a $k$-uniform hypercactus with $m$ edges and $r$ cycles. By Theorem~\ref{3}, $\rho_{\alpha}(G')> \rho_{\alpha}(G)$, a contradiction. Thus, every edge in $G$ has $k-2$ vertices of degree $1$.

Suppose that there exist two vertex-disjoint cycles. We choose two such cycles $C_1$ and $C_2$ by requiring that  $d_G(C_1,C_2)$ is as small as possible, where $d_G(C_1,C_2)=\min\{d_G(u,v): u\in V(C_1), v\in V(C_2)\}$. Let $u\in V(C_1)$ and $v\in V(C_2)$ with $d_G(C_1,C_2)=d_G(u,v)$. We may assume that $x_u\ge x_v$. Let $G''$ be the hypergraph  obtained from $G$ by moving edges containing $v$ in $C_2$ from $v$ to $u$. Obviously,  $G''$ is a $k$-uniform hypercactus with $m$ edges and $r$ cycles.
By Theorem~\ref{moving}, $\rho_{\alpha}(G'')> \rho_{\alpha}(G)$,  a contradiction. Thus, if $r\ge 2$, then all cycles in $G$ share a common vertex, which we denote by $w$. If $r=1$, then $w$ is a vertex of degree $2$ of the unique cycle.

Let $(v_0, e_1, v_1,\dots, v_{\ell-1}, e_{\ell}, v_0)$ be a cycle of $G$ of length $\ell\ge 2$, where $v_0=w$. Suppose that  $\ell\ge 3$. Assume that $x_{v_0}\ge x_{v_2}$. Let $G^*$ be the hypergraph  obtained from $G$ by moving the edge $e_2$ from $v_2$ to $v_0$.  Obviously,  $G^*$ is a $k$-uniform hypercactus with $m$ edges and  $r$ cycles. By Theorem~\ref{moving}, $\rho_{\alpha}(G^*)> \rho_{\alpha}(G)$,  a contradiction. Thus, every cycle of $G$ is of length $2$, and there are exactly $m-2r$ edges that are not on any cycle.

Suppose that $G\not\cong H_{m, r, k}$. Then there exists a vertex $z$ such that $d_G(w,z)=2$. Let $z'$ be the unique vertex such that $d_G(w,z')=d_G(z', z)=1$. There are two cases. First suppose that $z'$ lies on some cycle. Let $e_1$ and $e_2$ the the cycle containing $w$ and $z'$. Let $H$ be the hypergraph obtained from $G$ by moving all edges containing $z'$ except $e_1$ and $e_2$ from $z'$ to $w$ if $x_w\ge x_{z'}$, and the hypergraph obtained from $G$ by moving all edges containing $w$ except $e_1$ and $e_2$ from $w$ to $z$ otherwise.
Now suppose that $z'$ does not lie on any cycle. Let $e$ the the edge containing $w$ and $z'$. Let $H$ be the hypergraph obtained from $G$ by moving all edges containing $z'$ except $e$ from $z'$ to $w$ if $x_w\ge x_{z'}$, and the hypergraph obtained from $G$ by moving all edges containing $w$ except $e$ from $w$ to $z$ otherwise.
%
%
in either case, $H$ is a $k$-uniform hypercactus with $m$ edges and  $r$ cycles.
By Theorem~\ref{moving}, $\rho_{\alpha}(H)> \rho_{\alpha}(G)$,  a contradiction.
It follows that  $G\cong H_{m, r, k}$.
\end{proof}

\begin{Corollary}\label{max} Suppose that $k\ge 2$.\\
 (i) If $G$ is a $k$-uniform hypertree with $m\ge 1$ edges, then $\rho_{\alpha}(G)\le \rho_{\alpha}(S_{m,k})$ for $0\le \alpha<1$ with equality if and only if $G\cong S_{m,k}$. \\
(ii) If $G$ is a $k$-uniform unicyclic hypergraphs with $m\ge 2$ edges, then  $\rho_{\alpha}(G)\le\rho_{\alpha}(H_{m, 1, k})$ for $0\le \alpha<1$ with equality if and only if $G\cong H_{m, 1, k}$.
\end{Corollary}

The cases when $\alpha=0$ in Corollary~\ref{max}~(i) and (ii)  have been known in \cite{LSQ,FT}.

For $2\le d\le m$,  let $S_{m,d,k}$ be the $k$-uniform hypertree obtained from the $k$-uniform loose path $P_{d,k}=(v_0, e_1, v_1, \dots, v_{d-1}, e_d, v_d)$ by attaching $m-d$ pendant edges at $v_{\lfloor \frac{d}{2}\rfloor}$. Obviously, $S_{m,2,k}\cong S_{m,k}$.

\begin{Theorem}  \label{diameter} For $k\ge 2$, let $G$ be a $k$-uniform hypertree with $m$ edges and diameter $d\ge 2$. For $0\le \alpha<1$, we have $\rho_{\alpha}(G)\le\rho_{\alpha}(S_{m, d, k})$  with equality if and only if $G\cong S_{m, d, k}$.
\end{Theorem}

\begin{proof} It is trivial for $d=2$. Suppose that $d\ge 3$.

Let $G$ be a $k$-uniform hypertree with maximum $\alpha$-spectral radius among hypertrees with $m$ edges and diameter $d$.

Let $P=(v_0,e_1,v_1,\dots,e_d,v_{d})$ be a diametral  path of $G$. Let $x$ be the $\alpha$-Perron vector  of $G$.

\noindent{\bf Claim $1$.} Every edge of $G$ has at least $k-2$ vertices of degree $1$.

Suppose that there is at least one edge with at least three vertices of degree at least $2$. Let $f=\{u_1, \dots, u_k\}$ be such an edge. First suppose that  $f$ is not an edge on $P$. We may assume that  $d_G(u_1, P)=d_G(u_i, P)-1$ for $i=2, \dots, k$, where $d_G(w,P)=\min \{d_G(w,v): v\in V(P)\}$. Then $d_G(u_1)\ge 2$. We may assume that $d_G(u_i)\ge 2$ for $i=2, \dots r$ and $d_G(u_i)=1$ for $i=r+1, \dots, k$, where $3\le r\le k$.
Let $G'$ be the hypertree obtained  from $G$ by moving all edges containing   $u_{3},\dots,u_{r}$ except $f$  from $u_{3},\dots,u_{r}$ to $u_{1}$. Obviously, $G'$ is a hypertree with $m$ edges and diameter $d$.
By Theorem~\ref{3}, $\rho_{\alpha}(G')> \rho_{\alpha}(G)$, a contradiction. Thus  $f$ is an edge on $P$, i.e., $f=e_i$ for some $i$ with $2\le i\le d-1$.
Let $e_i\setminus\{v_{i-1}, v_i\} =\{v_{i,1},\dots,v_{i,k-2}\}$.
We may assume that $v_{i,1},\dots,v_{i,s}$ are precisely those vertices with degree at least $2$ among $v_{i,1},\dots,v_{i,k-2}$, where $1\le s\le k-2$.
Let $G''$ be the hypertree obtained from $G$ by moving
all edges containing   $v_{i,1},\dots,v_{i,s}$ except $e_i$ from $v_{i,1},\dots,v_{i,s}$ to $v_i$.  Obviously, $G''$ is a hypertree with $m$ edges and diameter $d$.  By Theorem~\ref{3}, $\rho_{\alpha}(G'')> \rho_{\alpha}(G)$, also a contradiction. It follows that all edges of $G$ have at most two vertices of degree at least $2$. Claim 1 follows.

\noindent{\bf Claim $2$.} Any edge not on $P$ is a pendant edge.

Suppose that $e$ is an edge not on $P$ and it is  not a pendant edge. Then there are two vertices,  say $u$ and $v$, in $e$ such that $d_u\ge 2$ and $d_v\ge 2$. Suppose without loss of generality that $d_G(u, P)<d_G(v, P)$. Let $w$ be the vertex on $P$ with  $d_G(u, P)=d_G(u, w)$.
Let $G^*$  be the hypertree obtained from $G$ by moving all edges containing $v$ except $e$ from $v$ to $w$ if $x_w\ge x_v$,  and  the hypertree obtained from $G$ by moving all edges containing $w$ (except the edge in the path connecting $w$ and $v$) from $w$ to $v$ otherwise. By Theorem~\ref{moving}, $\rho_{\alpha}(G^*)> \rho_{\alpha}(G)$,  a contradiction. This proves Claim 2.

\noindent{\bf Claim $3$.} There is at most one vertex of degree greater than two in $G$.

Suppose that there are two vertices, say $s$ and $t$, on $P$  with degree greater than two. We may assume that $x_s\ge x_t$. Let $H$ be the hypertree obtained from $G$ by moving all pendant edges containing  $t$ from $t$ to $s$. By Theorem~\ref{moving}, we have $\rho_{\alpha}(H)> \rho_{\alpha}(G)$, a contradiction. Claim 3 follows.

Combing Claims 1--3, $G$ is a hypertree obtained from the path $P$ by attaching $m-d$ pendant edges at some $v_i$ with $1\le i\le d-1$, and by Theorem~\ref{pq}, we have $G\cong S_{m, d, k}$.
\end{proof}

We mention that the above result for $\alpha=0$ has been proved in \cite{XW2} by a relation between the $0$-spectral radius of a power hypergraph and the $0$-spectral radius of its graph.

Suppose that $m\ge d\ge 3$. Let $H$ be the hypergraph obtained from $S_{m, d, k}$ by moving edge $e_d$ from $v_{d-1}$ to $v_{\lfloor \frac{d}{2}\rfloor}$ if $x_{v_{\lfloor \frac{d}{2}\rfloor}}\ge x_{v_{d-1}}$,
 and the hypergraph obtained from $S_{m, d, k}$ by moving edges containing $v_{\lfloor \frac{d}{2}\rfloor}$ except $e_{\lfloor \frac{d}{2}\rfloor+1}$  from $v_{\lfloor \frac{d}{2}\rfloor}$ to $v_{d-1}$ otherwise. Obviously, $H\cong S_{m, d-1, k}$. By Theorem~\ref{moving}, $\rho_{\alpha}(S_{m, d, k})<\rho_{\alpha}(S_{m, d-1, k})$. Now by Theorem~\ref{diameter}, Corollary~\ref{max}(i) follows. Moreover, if $G$ is a $k$-uniform hypertree
with $m\ge 3$ edges and $G\ncong S_{m,k}$,  $\rho_{\alpha}(G)\le \rho_{\alpha}(S_{m,3,k})$ with equality if and only if $G\cong S_{m,3,k}$, which
has been known for $\alpha=0$  in \cite{LSQ}.

For $2\le t\le m$, let $T_{m,t,k}$ be the $k$-uniform hypertree consisting of $t$ pendant paths of almost equal lengths (i.e., $t-\left(m-t\lfloor\frac{m}{t}\rfloor\right)$ pendant paths of length $\lfloor\frac{m}{t}\rfloor$ and
 $m-t\lfloor\frac{m}{t}\rfloor$ pendant paths of length $\lfloor\frac{m}{t}\rfloor+1$)
 at a common vertex.  Particularly, $T_{m,2,k}$ is just the $k$-uniform loose path $P_{m,k}$.

\begin{Theorem}  Let $G$ be a $k$-uniform hypertree with $m$ edges and $t\ge 2$ pendant edges. For $0\le \alpha<1$, we have $\rho_{\alpha}(G)\le\rho_{\alpha}(T_{m, t, k})$  with equality if and only if $G\cong T_{m, t, k}$.
\end{Theorem}

\begin{proof} Let $G$ be a $k$-uniform hypertree with maximum $\alpha$-spectral radius among hypertrees with $m$ edges and $t$ pendant edges. Let $x$ be the $\alpha$-Perron vector of $G$.

Suppose that there exists an edge $e=\{u_1, \dots, u_k\}$ with at least three vertices of degree at least $2$. Let $G'$ be the hypertree obtained  from $G$ by moving all edges containing  $u_{3},\dots,u_{k}$ except $e$  from these vertices to $u_{1}$. Obviously, $G'$ is a hypertree with $m$ edges and $t$ pendant edges.
By Theorem~\ref{3}, $\rho_{\alpha}(G')> \rho_{\alpha}(G)$, a contradiction. It follows that each edge of $G$ has at most two vertices of degree at least $2$.

Suppose that there are two vertices, say $u,v$  with degree greater than $2$. We may assume that $x_u\ge x_v$. Let  $H$ be the hypertree obtained from $G$ by moving an edge not on the path connecting $u$ and $v$ containing $v$ from $v$ to $u$. By Theorem~\ref{moving}, we have $\rho_{\alpha}(H)> \rho_{\alpha}(G)$, a contradiction. Thus, there is at most one vertex of degree greater than $2$ in $G$.

If there is no vertex of degree greater than $2$, then $t=2$, and $G$ is  the $k$-uniform loose path $P_{m,k}$. If there is exactly one vertex of degree greater than $2$, then $t\ge 3$,
$G$ is a hypertree consisting of $t$ pendant paths at a common vertex, and by Theorem~\ref{pq}, we have $G\cong T_{m, t, k}$.
\end{proof}

For $\alpha=0$, this is known in \cite{XW3,ZL}.

\bigskip

\noindent {\bf Acknowledgement.} This work was supported by the National Natural Science Foundation of China (No.~11071089).


\end{document}